\documentclass[11pt]{article}
\usepackage{latexsym}
\usepackage{tikz}
\newtheorem{theorem}{Theorem}

\newtheorem{proposition}[theorem]{Proposition}
\newtheorem{definition}{Definition}

\newtheorem{conjecture}[theorem]{Conjecture}

\newcommand{\qed}{\ \hfill\mbox{$\Box$}\vspace{\baselineskip}}
\newenvironment{proof}{\noindent {\bf Proof:}}{{\qed}}

\begin{document}
\renewcommand{\thetheoremvar}{\arabic{theorem}$'$}

\title{Graphs, Skeleta and Reconstruction of Polytopes}

\author{Margaret M. Bayer\thanks{This article is based in part on 
        work supported by the National Science Foundation
        under Grant No.\ DMS-1440140 while the author was in
        residence at the Mathematical Sciences Research Institute
        in Berkeley, California, and while supported by a University
        of Kansas sabbatical, during the Fall 2017 semester.}\\
        Department of Mathematics\\
        University of Kansas\\
        Lawrence, KS  66045-7594\\
        bayer@ku.edu}
\date{\relax}

\maketitle

\begin{quote}
Dedicated to Tibor Bisztriczky, G\'{a}bor Fejes T\'{o}th and Endre Makai,
on the occasion of their birthdays.
\end{quote}

\begin{abstract}
A renowned theorem of Blind and Mani, with a constructive proof by Kalai
and an efficiency proof by Friedman, 
shows that the whole face lattice of a simple polytope can be determined
from its graph.  This is part of a broader story of reconstructing face
lattices from partial information, first considered comprehensively
in Gr\"unbaum's
1967 book.  This survey paper includes varied results and open questions
 by many researchers 
on simplicial polytopes, nearly simple polytopes, cubical polytopes,
zonotopes, crosspolytopes, and Eulerian posets.

\end{abstract}

\section{Background}
This is a survey paper on reconstruction of polytopes, with an emphasis on
determining the face lattice of a polytope from its graph or from higher
dimensional skeleta.
We assume basic familiarity with the combinatorial theory of
convex polytopes.  For definitions the reader can consult 
Gr\"{u}nbaum \cite{grunbaum} or Ziegler \cite{ziegler}.
The reader is also directed to Kalai \cite{kalai-skel} for a survey of 
several topics on graphs and polytopes, including reconstruction.

How much combinatorial information is needed to determine the entire face
lattice of a convex polytope?
This is the subject of Chapter~12 of Gr{\"u}nbaum's book \cite{grunbaum}.
In some sense, little information is needed: the vertex-facet incidences
of the polytope determine the face lattice.
(In fact, for $d\ge 5$, the face lattice of a $d$-polytope can be reconstructed
from the incidences of edges and $(d-2)$-faces \cite[Exercise \#11 on
page 234]{grunbaum}.)
In another sense, a lot of information is needed: the bottom half of the
face lattice of a cyclic polytope is the same as the bottom part of a 
simplex of higher dimension.

In this paper all polytopes are real and convex.  We do not actually use
the embedding of the polytope in Euclidean space, but we restrict to those
that can be embedded.  We often blur the distinction between the polytope
and its face lattice.
Write $P\cong Q$ to mean that $P$ and $Q$ are combinatorially equivalent
polytopes, that is, their face lattices are isomorphic.

For $P$ a $d$-polytope and $k\le d-1$, the {\em $k$-skeleton} of $P$ is the
subcomplex of the boundary complex of $P$ consisting of all faces of $P$ of
dimension at most $k$.
Two polytopes are {\em $k$-equivalent} if their $k$-skeletons are
combinatorially equivalent.  The polytopes need not be of the same
dimension.

\begin{definition} {\em A $d$-polytope $P$  is {\em $k$-neighborly} if every
$k$-element subset of the vertices of $P$ is the vertex set of a face of $P$.
A $d$-polytope is {\em neighborly} if it is $\lfloor d/2 \rfloor$-neighborly.}
\end{definition}
Thus a $k$-neighborly polytope with $n+1$ vertices is $(k-1)$-equivalent to 
the $n$-simplex.

Cyclic polytopes are neighborly polytopes.  Thus cyclic
$d$-polytopes  are
$(\lfloor d/2 \rfloor -1)$-equivalent to higher dimensional polytopes.  In
general they can also be $(\lfloor d/2 \rfloor -1)$-equivalent to 
different $d$-polytopes.
Gr\"{u}nbaum \cite[7.2.4]{grunbaum} constructs
an example of a neighborly 4-polytope with 8 vertices that is not
combinatorially equivalent to the cyclic 4-polytope with 8 vertices.
The construction can be used to produce similar examples with more
vertices and in higher dimensions.
Padrol \cite{padrol} gives constructions of many neighborly polytopes.

The following results (due to Gale, Gr{\"u}nbaum, and Perles)
can be found in \cite{grunbaum} (in the original text or
in the additional notes).

\begin{theorem} \label{dim-ambig}

\mbox{}

\begin{enumerate}
\item
For $d\ge 2k+2$ the $d$-simplex is 
$k$-equivalent to the cyclic polytopes of dimensions $2k+2$, $2k+3$, \ldots, 
$d-1$ with $d+1$ vertices.
\item If $P$ and $Q$ are $k$-equivalent polytopes, $\dim(P)=d$, and 
$k\ge \lfloor d/2\rfloor$, then $\dim(Q)=d$.
\end{enumerate}
\end{theorem}

\begin{theorem}\label{polytope-rec}
Assume $P$ and $Q$ are $d$-polytopes.
\begin{enumerate}
\item If $P$ and $Q$ are $(d-2)$-equivalent, 
      then $P\cong Q$.\label{gen-equiv}
\item If $P$ and $Q$ are simplicial and $P$ and $Q$ are  
      $\lfloor d/2 \rfloor$-equivalent, 
      then $P\cong Q$.\label{simplicial-equiv}
\item If $P$ is a simplex and $P$ and $Q$ are 0-equivalent, 
      then $P\cong Q$.
\end{enumerate}
\end{theorem}

The last statement is, of course, a pedantic way of saying that
$d$-simplices are the only convex $d$-polytopes with exactly $d+1$ vertices.
Part~\ref{gen-equiv}  says that any $d$-polytope (or, more
precisely, its face lattice) can be reconstructed 
from its $(d-2)$-skeleton. 
In particular, for 
$d=3$, the graph of a 3-polytope determines the
face lattice;  this is a consequence of much earlier results by 
Steinitz \cite{steinitz} and Whitney \cite{whitney}. 
Dancis \cite{dancis} proves a result analogous to
Part~\ref{simplicial-equiv} for triangulated manifolds.
For simplicial $d$-polytopes, the face lattice can be reconstructed from
the incidences of $i$-faces and $(i+1)$-faces, if 
$\lfloor d/2\rfloor < i \le d-2$ (\cite[Theorem 19.5.25]{kalai-skel}).

Note in Part~\ref{simplicial-equiv} of Theorem~\ref{polytope-rec}, 
we need to assume both $P$ and $Q$ are simplicial.
The following 5-polytopes are 2-equivalent: 
the bipyramid over the 4-simplex, and 
the pyramid over the bipyramid over the 3-simplex.
The first polytope is simplicial, while the second is ``quasisimplicial''---all
its facets are themselves simplicial polytopes.

A much stronger (and surprising) result holds for simple polytopes, the
duals of simplicial polytopes.

\begin{theorem}[Blind and Mani \cite{MR89b:52008} and Kalai \cite{kalai}]
If $P$ and $Q$ are 1-equivalent simple $d$-polytopes, then $P$ and $Q$ are
combinatorially equivalent.
\end{theorem}

That is, the face lattice of a simple polytope is determined by its graph.
Note that the theorem assumes $P$ and $Q$ are of the same dimension.
The graph of a simple $d$-polytope may also be the graph of a nonsimple, lower
dimensional polytope.  For example, the graph of the simplex is also the graph
of lower dimensional cyclic polytopes; the graph of the cube is also the
graph of lower dimensional ``neighborly cubical polytopes''.
The dual statement of the theorem is that the face lattice of a simplicial 
polytope can be reconstructed from its facet-ridge graph.

Kalai's proof gives a method for constructing the face lattice from the graph,
but the complexity is exponential in the 
number of vertices---the algorithm uses all acyclic orientations of the graph.
It is straightforward, however, to construct the face lattice of a simple 
polytope from its 2-skeleton.  (See \cite{kalai}.)
Murty \cite{murty} shows that the 2-skeleton is enough to determine the
face poset of a broader class of objects, 
``simple'' {\em abstract polytopes}.
(An abstract or {\em incidence polytope}, introduced by Danzer and
Schulte \cite{danzer-schulte}, is a strongly flag-connected graded poset with
the diamond property.  Earlier, Adler \cite{adler} had given a similar 
definition, but with a further condition that each
vertex is contained in exactly $d$ facets.)
Joswig \cite{joswig-reconst} shows that the face lattice of any polytope $P$
can be
reconstructed from its graph along with the following additional information:
for each
vertex $v$ the sets of edges that are precisely the edges containing $v$ and
contained in a facet of $P$.

Various papers consider algorithmic issues of constructing the face
lattice of a simple polytope from its graph 
(\cite{achatz,joswig-ksystems,kaibel}; 
Friedman \cite{friedman} gives a polynomial time algorithm for finding the 
facets of the simple polytope from the graph.

Results of Perles show the limitations of polytope reconstruction
theorems, even for polytopes with few vertices: there are many more
combinatorial types of $d$-polytopes with $d+3$ vertices than combinatorial
types of $k$-skeleta of these polytopes.
\begin{theorem}[Perles, see \cite{grunbaum}]
The number of combinatorial types of $d$-polytopes with $d+3$ vertices 
is bounded below by an exponential function of $d$.  
\end{theorem}
\begin{theorem}[Perles, see Part II of \cite{Wotzlaw-diss}, also 
\cite{kalai-nato}]
For fixed $k$ and $b$, the number of combinatorial types of $k$-skeleta
of $d$-polytopes with $d+b+1$ vertices is bounded above by a constant
independent of dimension $d$.
\end{theorem}

\section{Nearly simple polytopes}

Recent papers by Doolittle, et al.\ \cite{doolittle} 
and by Pineda-Villavicencio, et al.\ \cite{pineda-v}
consider the possibility of extending the 
result on simple polytopes to polytopes that have few nonsimple vertices
(vertices of degree greater than the dimension of the polytope).

\begin{theorem}[\cite{doolittle,pineda-v}]\label{nearly-simple}

\mbox{}

\begin{enumerate}\label{doo-nevo}
\item The face lattice of any $d$-polytope with at most two nonsimple
      vertices is determined by its graph. \label{ns-pt1}
\item If $d\ge 5$ and $n\le 2d$, then the face lattice of any $d$-polytope
      with $n$ vertices, at most $2d-n+3$ of them nonsimple, is determined
      by its graph.
\item The face lattice of any $d$-polytope with at most $d-2$ nonsimple
      vertices is determined by its 2-skeleton.
\item For every $d\ge 4$, there are two combinatorially nonequivalent $d$-polytopes
      with $d-1$ nonsimple vertices and isomorphic $(d-3)$-skeleton.
      \label{noniso}
\end{enumerate}
\end{theorem}

In terms of $k$-equivalence, this says:
\begin{quote}
If a $d$-polytope $P$ has at most two nonsimple vertices, and a $d$-polytope
$Q$ is 1-equivalent to $P$, then $P\cong Q$.\\
If a $d$-polytope $P$ has at most $d-2$ nonsimple vertices, and a $d$-polytope
$Q$ is 2-equivalent to $P$, then $P\cong Q$.\\
There are $(d-3)$-equivalent $d$-polytopes with $d-1$ nonsimple
vertices that are not combinatorially equivalent.
\end{quote}

The paper \cite{doolittle} presents two proofs of Part~\ref{ns-pt1} of the 
theorem.  One uses the acyclic orientation approach of  Kalai's proof 
\cite{kalai} of the simple case and the ``frames'' of Joswig, Kaibel and 
K\"{o}rner \cite{joswig-ksystems}.  (See also \cite{friedman,kaibel}.)
The other uses acyclic orientations along with
truncation of the nonsimple vertices.

The construction in \cite{doolittle}
of Part~\ref{noniso} of Theorem~\ref{doo-nevo} 
gives a 4-polytope with $f$-vector $(8, 19, 18, 7)$, one of whose facets is 
the bipyramid over a triangle.  Splitting this facet into two simplices
gives a complex combinatorially equivalent to a 4-polytope with 
$f$-vector $(8, 19, 19, 8)$.  The construction does not change the 
graph and therefore does not change the number (three) of nonsimple
vertices.

Pineda-Villavicencio et al.\ \cite{pineda-v,pineda-excess} also consider another
measure of deviation from simple.  The {\em excess degree} of a $d$-polytope
is $2f_1-df_0$, the sum of the number of extra edges on all the vertices.
They study Minkowski decomposability and prove some structural properties of
polytopes with small excess.  On the issue of reconstruction,
they prove the following theorem.

\begin{theorem}[\cite{pineda-v}]
The face lattice of any $d$-polytope with excess degree at most $d-1$ is
determined by its graph.
\end{theorem}

\section{Crosspolytopes and Centrally Symmetric Polytopes}

Espenschied \cite{espenschied} asked for the dimensions of polytopes that are
1-equivalent to the $d$-crosspolytope.
The $d$-dimensional crosspolytope can be obtained by starting with the
1-polytope and successively taking bipyramids $d-1$ times.
The graph of the $d$-crosspolytope is the complete $d$-partite graph with
2 vertices in each part, denoted $K_{2,2,\ldots,2}$.
This graph does not uniquely determine the polytope, in general.
For example, one can take the convex hull of the 3-crosspolytope in ${\bf R}^4$ 
along with a segment intersecting the 3-crosspolytope only at the interior point
of one of its facets.  The result is a 4-polytope 1-equivalent to the 
4-crosspolytope, having one facet a triangular bipyramid.
In what dimensions other than $d$ do there exist
polytopes $1$-equivalent to the $d$-crosspolytope?

\begin{theorem}[Espenschied \cite{espenschied}]
\label{espens}
Let $d$ be an integer, $d\ge 4$.
\begin{enumerate}
\item If $P$ is a polytope that is 1-equivalent to the
      $d$-crosspolytope, then $\dim(P)\le \lfloor 3d/2 \rfloor - 1$.
      \label{cross1}
\item \label{cross2}
      For every $n$, $d\le n \le \lfloor 3d/2 \rfloor - 1$, there exists
      an $n$-dimensional polytope that is 
      1-equivalent to the $d$-crosspolytope.  
\item If $d\ge 5$, there exists a $(d-1)$-dimensional polytope that is
      1-equivalent to the $d$-crosspolytope.  \label{cross3}
\end{enumerate}
\end{theorem}
(Part~\ref{cross3} is derived from an example of \cite{grunbaum}; see below.)

Part~\ref{cross2}  generalizes to $k$-equivalence.
\begin{theorem}\label{cross-k}
Let $k$ be a positive integer, $d\ge 2(k+1)$.
For every $n$, $d\le n\le \lfloor (\frac{k+2}{k+1})d \rfloor -1$, there exists
an $n$-dimensional polytope that is $k$-equivalent to the $d$-crosspolytope.
\end{theorem}

The proof of 
Part~\ref{cross2} of Theorem~\ref{espens} follows from 
Espenschied's observation  that by taking joins of
a sequence of crosspolytopes, one obtains polytopes 1-equivalent to a
crosspolytope.  The idea extends to give Theorem~\ref{cross-k} as well.
\begin{definition}{\em The {\em join} $P*Q$ of $d$-polytope $P$ and $e$-polytope
$Q$ is the convex hull of copies of $P$ and $Q$ embedded 
in skew affine subspaces of ${\bf R}^{d+e+1}$.}
\end{definition}

Denote the $j$-dimensional crosspolytope by $X_j$.
For integers $j_1, j_2, \ldots, j_s$, $j_i\ge k+1$, let 
$d=j_1+j_2+\cdots +j_s$.
The join
$X_{j_1}*X_{j_2}*\cdots *X_{j_s}$ is a polytope of dimension
$d+s-1$, and is $k$-equivalent to the crosspolytope
$X_d$.
This construction gives Theorem~\ref{cross-k} and, in particular, 
Part~\ref{cross2} of Theorem~\ref{espens}.
Note that the graph of the crosspolytope is $d$-colorable, so any 
$n$-polytope with $n>d$ that is 1-equivalent to $X_d$ can have no 
simplex faces of dimension greater than $d-1$.

The regular $d$-crosspolytope is centrally symmetric; it is clearly the
centrally symmetric $d$-polytope with the fewest vertices.
Gr\"{u}nbaum \cite[Sec.\ 6.4]{grunbaum} gives the following example of a 
centrally symmetric 
4-polytope with 10 vertices and 40 edges.  
This 4-polytope has vertices
$\pm e_1$, $\pm e_2$, $\pm e_3$, $\pm e_4$,
\mbox{$\pm (e_1+e_2+e_3+e_4)$}, and has the graph of the 5-crosspolytope.
By taking successive bipyramids over Gr\"{u}nbaum's 4-polytope, one obtains
the polytopes of Theorem~\ref{espens}, Part~\ref{cross3}.
Espenschied observes that the join of 
Gr\"{u}nbaum's 4-polytope with a square is a 
 7-dimensional polytope that is 1-equivalent, but not combinatorially
equivalent, to the 7-crosspolytope; it is, after all, nonsimplicial.

A centrally symmetric polytope $P$ is {\em centrally symmetric $k$-neighborly}
if every set of $k$ vertices of $P$, no two of which are antipodes, is
the vertex set of a $(k-1)$-face of $P$.
The $d$-crosspolytope is centrally symmetric $k$-neighborly for all $k\le d$.
Thus a centrally symmetric polytope with $2d$ vertices is
$k$-neighborly if and only if it
is $(k-1)$-equivalent to the $d$-crosspolytope.
McMullen and Shephard \cite{mcmullen} 
introduced centrally symmetric diagrams (analogues of
Gale diagrams) to study centrally symmetric polytopes with few vertices.
They used these to construct $d$-polytopes that have $2(d+2)$ vertices and are 
centrally symmetric $k$-neighborly, for $k\approx d/3$.
Barvinok, Lee and Novik \cite{barvinok} construct centrally symmetric
$k$-neighborly $n$-polytopes where $n$ is small relative to the number of
vertices, that is, small relative to the dimension $d$ of the $(k-1)$-equivalent
$d$-crosspolytope.  In particular, they give polytopes of dimension approximately
$2\log_3(2d)$, 1-equivalent to the $d$-crosspolytope.
See also \cite{bpsz} for construction of centrally symmetric neighborly spheres.

Joswig and Ziegler \cite{joswig-ziegler} ask if for every $n$, $4\le n\le d$
and
$k=\lfloor n/2\rfloor -1$ there is an $n$-polytope $k$-equivalent to the
$d$-crosspolytope.
Linial and Novik \cite{linial-novik} prove the existence
of $k$-neighborly centrally symmetric $n$-polytopes with $2d$ vertices, 
that is, $n$-polytopes $(k-1)$-equivalent to the $d$-crosspolytope,
with asymptotic estimates for $k$ in terms of $n$ and $d$. 
(See also Donoho \cite{donoho2}.)

Recall that the graph of the crosspolytope is the complete multipartite graph
with all parts of size 2.  What other complete multipartite graphs are the 
graphs of polytopes?

\begin{proposition}[Espenschied \cite{espenschied}]
For $1\le i\le t$, let $n_i\in \{1,2\}$.  
The complete multipartite graph $K_{n_1,n_2,\ldots, n_t}$ ($t>1$) is the graph
of a polytope if and only if 
$\{n_1,n_2,\ldots, n_t\}$ is not one of the multisets $\{1,2\}$ or $\{1,1,2\}$.
\end{proposition}
These graphs can all be realized with iterated pyramids and bipyramids.

Espenschied guessed that these were the only complete multipartite graphs
realizable as graphs of polytopes.
It is easy to check, using the nonplanarity of $K_{3,3}$, that $K_{m,n}$ is not
the graph of any polytope, when $m$ and $n$ are both at least 3 \cite{barnette}.
However, Firsching \cite{firsching} has found examples showing Espenschied's 
guess is incorrect.  
He has examples of 4-polytopes with 9 vertices with the following graphs: 
$K_{3,2,2,2}$, $K_{3,2,2,1,1}$, $K_{3,2,1,1,1,1}$ and $K_{3,1,1,1,1,1,1}$.
Zheng \cite{zheng} constructs simplicial 3-spheres with graph 
$K_{4,4,4,4}$, but they are believed not to be polytopal.

\section{Cubical Polytopes and Zonotopes}

Perhaps the first thing to note about cubical polytopes is that the results
about crosspolytopes can be dualized to results about cubes.  
A polytope 1-equivalent to a $d$-crosspolytope then corresponds to a polytope
that shares the facet-ridge graph with a cube.

\begin{definition}
{\em A polytope is {\em cubical} if and only if all of its proper faces are
combinatorially equivalent to cubes.}
\end{definition}

\begin{theorem}[Joswig and Ziegler \cite{joswig-ziegler}]
\mbox{}
\begin{enumerate}
\item
For $n\ge d\ge 2k+2$, there exists a cubical $d$-polytope $k$-equivalent
to the $n$-dimensional cube.
\item If a $d$-polytope $P$ is $k$-equivalent to the $n$-cube for $k\ge d/2$,
then $P$ is a $d$-cube.
\end{enumerate}
\end{theorem}

Cubical $d$-polytopes that are $(\lfloor d/2\rfloor -1)$-equivalent to an
$n$-dimensional cube (for some $n\ge d$) are called {\em neighborly cubical}
polytopes.
A neighborly cubical $d$-polytope may not be reconstructible from its
$(\lfloor d/2\rfloor -1)$-skeleton, however.
Joswig and Ziegler \cite{joswig-ziegler} give an example of a 4-polytope that
is 1-equivalent to a 5-cube (and thus to a neighborly cubical 4-polytope), but
that is not itself cubical.  (It has a facet that can be subdivided into two
3-cubes.)
See also \cite{bbc,joswig-rorig}.

Joswig \cite{joswig-reconst} shows that a certain class of cubical polytopes,
``capped cubical polytopes,'' can be reconstructed from their graphs.

Another generalization of the cube is the zonotope, the Minkowski sum of 
1-polytopes.  

\begin{theorem}
\mbox{}
\begin{enumerate}
\item \textbf{{\em (Bj\"{o}rner, Edelman and Ziegler \cite{bez})}}
      Zonotopes are determined by
      their graphs.
\item \textbf{{\em (Babson, Finschi and Fukuda \cite{bff})}}
      Duals of cubical zonotopes are determined by their graphs.
\end{enumerate}
\end{theorem}

\section{Eulerian Posets}

In the previous sections we have seen conditions under which skeletal 
information about a polytope enables us to reconstruct the entire face
lattice.  What if we consider face lattices of polytopes within the larger
class of Eulerian posets?  

\begin{definition} {\em An {\em Eulerian poset} is a graded, finite partially
ordered set such that in each interval, the number of
elements of even rank equals the number of elements of odd rank.}
\end{definition}

The face lattice of a $d$-polytope is an Eulerian poset of rank $d+1$.
In the mixed company of polytopes and posets, there is always the difficulty
of choosing between dimensions and ranks.
Here the dimension perspective will prevail.
For a rank $d+1$ Eulerian poset $Q$, write $Q_k$ for the set of rank $k+1$
elements (also called dimension $k$ elements).

Suppose we change the hypothesis of Theorem~\ref{polytope-rec} from 
``Assume $P$ and $Q$ are $d$-polytopes'' to ``Assume $P$ is a 
$d$-polytope and $Q$ is a rank $d+1$ Eulerian poset.''  Do we get a 
theorem?  No, not even for Part~3 of the theorem.  Figure~1
shows an Eulerian poset of rank~4 with four atoms that is not the Boolean
algebra (face lattice of a 3-simplex).  It also serves to illustrate the
proof of Theorem~\ref{Eulerian}.

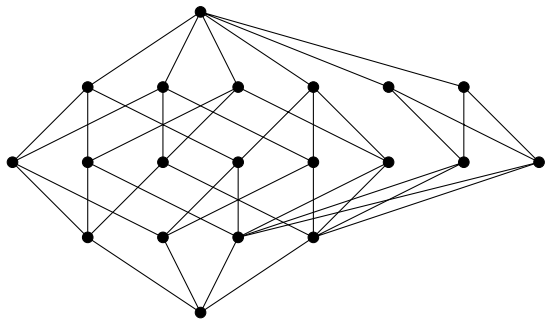
\begin{figure}[h]\label{not-simplex}
\centering

\begin{tikzpicture}
\draw (3.5,0) --(2,1) --(1,2) --(2,3)--(3.5,4)--(3,3)--(1,2)--(3,1);
\draw (3.5,0)--(3,1)--(4,2)--(2,3)--(2,2)--(2,1)--(3,2)--(3,3)--(5,2);
\draw (3.5,0)--(4,1)--(2,2)--(4,3)--(3,2)--(5,1)--(5,2)--(3,1);
\draw (3.5,0)--(5,1)--(6,2)--(4,3)--(3.5,4)--(5,3)--(4,2)--(4,1);
\draw (4,1)--(6,2)--(5,3)--(5,2);
\draw (4,1)--(7,2)--(6,3)--(3.5,4)--(7,3)--(7,2)--(5,1)--(8,2)--(6,3);
\draw (4,1)--(8,2)--(7,3);
\filldraw[black] (3.5,0) circle (2pt);
\filldraw[black] (2,1) circle (2pt);
\filldraw[black] (3,1) circle (2pt);
\filldraw[black] (4,1) circle (2pt);
\filldraw[black] (5,1) circle (2pt);
\filldraw[black] (1,2) circle (2pt);
\filldraw[black] (2,2) circle (2pt);
\filldraw[black] (3,2) circle (2pt);
\filldraw[black] (4,2) circle (2pt);
\filldraw[black] (5,2) circle (2pt);
\filldraw[black] (6,2) circle (2pt);
\filldraw[black] (7,2) circle (2pt);
\filldraw[black] (8,2) circle (2pt);
\filldraw[black] (2,3) circle (2pt);
\filldraw[black] (3,3) circle (2pt);
\filldraw[black] (4,3) circle (2pt);
\filldraw[black] (5,3) circle (2pt);
\filldraw[black] (6,3) circle (2pt);
\filldraw[black] (7,3) circle (2pt);
\filldraw[black] (3.5,4) circle (2pt);

\end{tikzpicture}
\caption{An Eulerian poset 0-equivalent to the 3-simplex}
\end{figure}

This generalizes to a construction of an Eulerian poset of rank $d+1$
that is $(d-3)$-equivalent to the $d$-simplex.
In fact, it generalizes much further.

\begin{theorem}\label{Eulerian}
If $P$ is a $d$-polytope, $d\ge 3$, and  $j$ is an integer, $0\le j\le d-1$,
then there exists a rank $d+1$ Eulerian poset, not isomorphic to the face
lattice of $P$, but which differs from the face lattice of $P$ only in 
dimensions $j$ and $j+1$.
\end{theorem}

\begin{proof}
Choose a $j$-face $F$ of $P$ and a $(j+1)$-face $G$ containing $F$.
Construct a poset $Q$ whose elements are all the elements of the face 
lattice of $P$ along with four new elements $A_1$, $A_2$, $B_1$ and $B_2$.
The pairs in the poset are those pairs in the face lattice of $P$ along 
with the following pairs: 
\begin{itemize}
\item $(A_i, B_\ell)$ for $i, \ell \in \{1,2\}$
\item $(C, A_i)$ and $(C,B_\ell)$, for $i, \ell \in \{1,2\}$ and $C$ a proper face of
      $F$
\item $(A_i, D)$ and $(B_\ell,D)$ for $i, \ell \in \{1,2\}$ and $D$ a face of $P$
      properly containing $G$.
\end{itemize}
(See Figure~1.)
By construction, the elements $A_i$ are of dimension $j$, the elements $B_\ell$
are of dimension $j+1$, and 
$Q$ agrees with the face lattice of $P$ except at dimensions
$j$ and $j+1$.  It is straightforward to check the interval condition to
show that  $Q$ is Eulerian.  We
illustrate with one of the more interesting cases, an interval of the 
form $[C,B_1]$, where $C$ is a proper face of  $F$.  Since the face lattice
of $P$ is Eulerian, the interval $[C,F]$ in $P$ has the same number of
elements of even and odd rank.  The interval $[C,B_1]$ of $Q$ has all
elements of $[C,F]$ except $F$, has two other elements ($A_1$ and $A_2$)  
of the same rank as $F$, and has one other element ($B_1$) of rank one more.
So the interval $[C,B_1]$ of $Q$ has the same number of elements of even 
and odd rank.
\end{proof}

Thus we know that Parts~\ref{simplicial-equiv} and~3 of Theorem~\ref{polytope-rec} fail in the
generality of Eulerian posets.  Part~\ref{gen-equiv} fails even for $P$ a 
simplicial polytope (but not for $P$ a simplex) and $Q$ an Eulerian poset.
Consider the $d$-crosspolytope.  The $2^d$ facets are naturally partitioned
into two sets, so that no two facets in the same set intersect at a ridge
($(d-2)$-face).  Remove one of the sets of facets, and duplicate the other
set.
The resulting poset is an Eulerian poset with the same $(d-2)$-skeleton
as the crosspolytope.

However, in the case of a simplicial polytope, we can reconstruct a
single missing rank of lower rank (dimension less than $d-1$).
Write $P\setminus P_r$ for the subposet of $P$ consisting of all the
elements of ranks other than $r+1$ (dimension $r$).

\begin{theorem}
Let  $d\ge 3$, and assume 
$P$ is a $d$-polytope and $Q$ is a rank $d+1$ Eulerian poset.
\begin{enumerate}
\item If $P$ is simplicial, $0\le r \le d-2$, and $P\setminus P_r\cong 
      Q\setminus Q_r$, then $P\cong Q$.
\item If $P$ is simple, $1\le r \le d-1$, and $P\setminus P_r\cong 
      Q\setminus Q_r$, then $P\cong Q$.
\item If $P$ is a simplex, $0\le r \le d-1$, and $P\setminus P_r\cong 
      Q\setminus Q_r$, then $P\cong Q$.
\end{enumerate}
\end{theorem}

\section{Related Issues on Graphs and $k$-Skeleta}

We review briefly some other work on graphs and $k$-skeleta of polytopes.

A relatively easy way of reconstructing simple polytopes from their graphs
would have followed from the truth of a conjecture by Perles:
The facet subgraphs of a simple $d$-polytope are exactly all the 
$(d-1)$-regular, connected, induced, nonseparating subgraphs of the 
graph of the polytope.  However, the conjecture is false (even when the
subgraph is required to be $(d-1)$-connected), as shown by Haase and
Ziegler \cite{haase-ziegler}.  In the known 4-dimensional counterexamples
the offending subgraphs are not planar, so they ask if adding planarity is
enough to guarantee a facet of a 4-polytope.  This extends to higher dimensions 
by asking for the following.  Are the facet subgraphs of a simple $d$-polytope
exactly all the $(d-1)$-regular, induced, nonseparating subgraphs that are
isomorphic to the graph of a $(d-1)$-polytope?  However, since there is no easy
characterization of graphs of higher dimensional polytopes, this would not be
an effective characterization.  Adiprasito, Kalai and Perles \cite{akp}
suggested that  ``isomorphic to the graph of a $(d-1)$-polytope'' could be 
weakened to ``isomorphic to the graph of a homology $(d-2)$-sphere.''

We have seen definitions of neighborly, centrally symmetric neighborly, and
neighborly cubical polytopes.  All can be defined in terms of
$k$-equivalence with familiar polytopes of higher dimensions.  
Here we mention two other such classes of polytopes.
A {\em prodsimplicial-neighborly} polytope \cite{matschke} is a 
polytope with the same $k$-skeleton (for specified $k$) as the Cartesian
product of simplices.
Matschke et al.\ \cite{matschke} construct prodsimplicial-neighborly polytopes,
including polytopes of dimension $2k+r+1$ that are $k$-equivalent to
the product of $r$ simplices.  The latter are Minkowski sums of cyclic
polytopes, and are examples of {\em Minkowski neighborly} polytopes, studied in
\cite{adiprasito}.

Several examples mentioned above (including cubical neighborly
polytopes) are created by projection from higher dimensional polytopes.
As another example, Sanyal and Ziegler \cite{sanyal} construct $d$-dimensional
projections of polytopes that are $(\lfloor d/2\rfloor-1)$-equivalent
to the $r$-fold product of $m$-gons, for every even $m\ge 4$ and every
$d\le 2r$.
However, in general it is not clear when the
projection of a polytope preserves its $k$-skeleton.
R\"{o}rig and Sanyal \cite{rorig} study obstructions to the existence of
projections of polytopes that preserve $k$-skeleta.

So far we have always assumed that the graph or $k$-complex under consideration
is known to be the graph or $k$-skeleton of a polytope.  
However, we have no characterization of graphs of polytopes of 
dimensions four and higher. 
Pfeifle, Pilaud and Santos \cite{pfeifle} review necessary conditions for 
a graph to be polytopal, construct families of graphs that satisfy these
conditions but are not polytopal, and investigate polytopality of products
of graphs.  In particular, they show the following theorem.

\begin{theorem}[\cite{pfeifle}]
The Cartesian product of graphs is
the graph of a simple polytope if and only 
each of its factors is the graph of a simple polytope.
\end{theorem}

\section{Open Problems}

We are left with many open problems.
The first is a conjecture of Gr\"{u}nbaum.

\begin{conjecture}[Gr\"{u}nbaum \cite{grunbaum}]
If a $k$-dimensional complex $\cal C$ is the $k$-skeleton of both a 
$d$-polytope and a $d''$-polytope, where $d\le d''$, then for every 
$d'$, $d\le d'\le d''$, there is a $d'$-polytope having $\cal C$ as its
$k$-skeleton.
\end{conjecture}

Of course, we can dream and ask: for each $k$, characterize the polytopes 
whose $k$-skeletons determine the face lattice.  Here are some less
ambitious questions.

\begin{itemize}
\item (\cite{doolittle}) For $d\ge 5$ does there exist an integer $j$ such
      that every $d$-polytope with $j$ nonsimple vertices is determined by
      its 2-skeleton, but not every $d$-polytope with $j$ nonsimple vertices
      is determined by its 1-skeleton?
\item For $k>1$ does there exist an $n$-dimensional polytope that is 
      $k$-equivalent to
      the $d$-crosspolytope for $\lfloor (\frac{k+2}{k+1})d \rfloor\le n
      \le  \lfloor 3n/2\rfloor -1$?
      The proof of the upper bound $\lfloor 3n/2\rfloor-1$ for $k=1$ 
      (Part~(\ref{cross1}) of Theorem~\ref{espens}) is based on the fact 
      (proved by Halin \cite{halin}) that the Hadwiger number of the graph
      $K_{2,2,\ldots,2}$ is 
      $\lfloor 3n/2\rfloor$.  Is there an analogous theory for $k>1$?
\item Does there exist a $d$-polytope $P$ (not simple or simplicial),
      an Eulerian poset $Q$, and an $r$, $1\le r\le d-2$, such that
      $Q$ agrees with the face lattice of $P$ everywhere except
      at dimension $r$?
      Note that the number of elements of $P_r$ is determined by Euler's 
      formula.
\item It seems that for any reasonable reconstruction results involving  
      polytopes and Eulerian posets we will need to restrict to Eulerian
      lattices.  What are the best reconstruction results in this case?
\item What reconstruction results can we get for the abstract
      polytopes of Danzer and Schulte \cite{danzer-schulte}?
      See \cite[Section 7]{schulte} for a discussion of problems about
      skeleta of abstract polytopes.
\end{itemize}

\vspace{12pt}

\noindent \textbf{Acknowledgments}

\vspace{6pt}

Thanks to Karim Adiprasito, Joseph Doolittle, Eran Nevo, Isabella Novik, 
Raman Sanyal, Hailun Zheng and G\"{u}nter Ziegler for close reading of a 
draft and useful comments.

\end{document}